\documentclass[a4paper,12pt]{amsart}

\usepackage{epsfig}
\usepackage{amsthm,amsfonts}
\usepackage{amssymb,graphicx,color}
\usepackage[all]{xy}
\usepackage{cancel} 
\usepackage{verbatim}
\usepackage{hyperref}
\usepackage{enumitem}

\usepackage[a4paper,top=2.5cm,bottom=2.5cm,left=3cm,right=3cm]{geometry}

\newtheorem{theorem}{Theorem}[section]
\newtheorem*{theorem*}{Theorem}
\newtheorem{lemma}[theorem]{Lemma}
\newtheorem{corollary}[theorem]{Corollary}
\newtheorem{proposition}[theorem]{Proposition}
\newtheorem{definition}[theorem]{Definition}
\newtheorem{example}[theorem]{Example}
\newtheorem{remark}[theorem]{Remark}

\newcommand{\C}{\mathbb C}
\newcommand{\R}{\mathbb R}

\newcommand{\N}{\mathbb N}

\begin{document}

\title[On Lipschitz rigidity of complex analytic sets]
{On Lipschitz rigidity of complex analytic sets}

\author[Alexande Fernandes]{Alexandre Fernandes}
\author[J. Edson Sampaio]{J. Edson Sampaio}

\address[Alexandre Fernandes]{ Departamento de Matem\'atica, Universidade Federal do Cear\'a,
	      Rua Campus do Pici, s/n, Bloco 914, Pici, 60440-900, 
	      Fortaleza-CE, Brazil. \newline  
              E-mail: {\tt alex@mat.ufc.br}
} 
\address[J. Edson Sampaio]{ BCAM - Basque Center for Applied Mathematics,
	      Mazarredo, 14 E48009 Bilbao, Basque Country - Spain.   
	      E-mail: {\tt esampaio@bcamath.org} \newline	      
              and       \newline    
              Departamento de Matem\'atica, Universidade Federal do Cear\'a,
	      Rua Campus do Pici, s/n, Bloco 914, Pici, 60440-900, 
	      Fortaleza-CE, Brazil. \newline  
              E-mail: {\tt edsonsampaio@mat.ufc.br}
}

\keywords{Lipschitz regularity, Tangent cone at infinity, algebraic sets}
\subjclass[2010]{14B05; 32S50}
\thanks{The first named author was partially supported by CNPq-Brazil grant 302764/2014-7.
The second named author was supported by the ERCEA 615655 NMST Consolidator Grant and also by the Basque Government through the BERC 2014-2017 program and by Spanish Ministry of Economy and Competitiveness MINECO: BCAM Severo Ochoa excellence accreditation SEV-2013-0323.
}

\begin{abstract}
We prove that any complex analytic set in $\mathbb{C}^n$ which is Lipschitz normally embedded at infinity and has tangent cone at infinity that is a linear subspace of $\mathbb{C}^n$ must be an affine linear subspace of $\mathbb{C}^n$ itself. No restrictions on the singular set, dimension nor codimension are required. In particular, a complex algebraic set in $\mathbb{C}^n$ which is Lipschitz regular at infinity is an affine linear subspace.
\end{abstract}

\maketitle

\section{Introduction}
Local Lipschtz geometry of complex algebraic sets has been intensively studied in the last years. One of the recent works on this subject, the paper \cite{BirbrairFLS:2016} (see also \cite{Sampaio:2016}),  somehow, showed a kind of rigidity of such a local geometric structure of algebraic sets. Indeed, it was proved  that Lipschitz regular complex algebraic germs of sets in $\C^n$, that is, germs of complex algebraic sets in $\C^n$ which are bi-Lipschitz homeomorphic to the germ of $\R^d$ at some point, are analytically smooth (see Theorem 3.2 in \cite{BirbrairFLS:2016} and Theorem 4.2 in \cite{Sampaio:2016}). From another way, looking to scrutinize global Lipschitz geometry of such sets in some sense, we arrived on the notion of subsets of $\C^n$ to be Lipschitz regular at infinity,  that means, subsets that outside a compact subset are bi-Lipschitz homeomorphic to the complement of a Euclidean ball in some $\R^d$.

The aim of this paper is to provide a proof that any complex algebraic set in $\mathbb{C}^n$ which is Lipschitz normally embedded at infinity and its tangent cone at infinity is a linear subspace of $\mathbb{C}^n$ must be an affine linear subspace of $\mathbb{C}^n$. In particular, complex algebraic subsets of $\C^n$ which are Lipschitz regular at infinity are affine linear subspace of $\C^n$. The main ingredients of these proofs stand on the notion of tangent cone at  infinity and the inner distance on path connected Euclidean subsets which we are going to address in the sections \ref{cones} and \ref{inner distance} respectively.

Let us observe that these results are somehow a Liouville or Bernstein type theorem as, for instance, a celebrated theorem due to Bombieri, De Giorgi and Miranda which says that entire positive minimal graph of functions in Euclidean spaces must be horizontal affine hyperplane. Finally, we get an application of the main results which is somehow  like  Theorem J in ([5], p. 180) (see also
[11] and [7]), namely, we prove that pure dimension analytic subsets of $\C^n$ which is  Lipschitz Normally Embedded and with tangent cone at infity being a linear subspace must be linear subspaces of $\C^n$ themself. Notice that, we address complex analytic sets in $\mathbb{C}^n$ which are not necessarily graph of smooth functions; a priori, they are not supposed  even smooth.

\bigskip


\section{Preliminaries}\label{preliminaries}
All the subsets of $\R^n$ (or $\C^n$) are considered equipped with the induced Euclidean metric.

\begin{definition}
Let $X\subset\R^n$ and $Y\subset\R^m$. A mapping $f\colon X\rightarrow Y$ is called {\bf Lipschitz} if there exists $\lambda >0$ such that is
$$\|f(x_1)-f(x_2)\|\le \lambda \|x_1-x_2\|$$ for all $x_1,x_2\in X$. A Lipschitz mapping $f\colon X\rightarrow Y$ is called
{\bf bi-Lipschitz} if its inverse mapping exists and is Lipschitz.
\end{definition}

\begin{definition}
Let $X\subset \R^n$ and $Y\subset\R^m$ be two subsets. We say that $X$ and $Y$ are {\bf bi-Lipschitz homeomorphic at infinity}, if there exist compact subsets $K\subset\R^n$ and $\widetilde K\subset \R^m$ and a bi-Lipschitz homeomorphism $\phi \colon X\setminus K\rightarrow Y\setminus \widetilde K$.
\end{definition}

\begin{definition}
A subset $X\subset\R^n$ is called {\bf Lipschitz regular at infinity} if $X$ and $\R^k$ are bi-Lipschitz homeomorphic at infinity, for some $k\in\N$.
\end{definition}

\begin{example}\label{real_example} Let $X\subset\R^3$ be defined by $ X = \{ (x,y,z)\in\R^3 \ : \ x^2+y^2=z^3 \}.$ We see that $X$ is an algebraic subset of $\R^3$ with an isolated singularity at $0\in\R^3.$ By using the mapping $\pi\colon X\rightarrow\R^2$; $\pi (x,y,z)=(x,y)$, it is easy to see that $X$ is Lipschitz regular at infinity.
\end{example}

\begin{example} Let $Y\subset\R^3$ be defined by $ Y = \{ (x,y,z)\in\R^3 \ : \ x^2+y^2=z \}.$ We see that $Y$ is a smooth algebraic subset of $\R^3$. From another way,  $Y$ is not Lipschitz regular at infinity. For instance, one can use Theorem \ref{tg_cones} to see that $Y$ is not Lipschitz regular at infinity.
\end{example}

At this moment we are ready to state one of the main results of the paper.

\begin{theorem}\label{main theorem}
Let $X\subset\C^n$ be a complex algebraic subset. If $X$ is Lipschitz regular at infinity, then $X$ is an affine linear subspace of $\C^n$.
\end{theorem}

We are going to prove this theorem in the Section \ref{main}. Notice that Theorem \ref{main theorem} does not hold true (with same assumptions) for real algebraic sets (cf. Example \ref{real_example} above).

\section{Inner distance}\label{inner distance}
Given a path connected subset $X\subset\R^m$ the
\emph{inner distance}  on $X$  is defined as follows: given two points $x_1,x_2\in X$, $d_X(x_1,x_2)$  is the infimum of the lengths of paths on $X$ connecting $x_1$ to $x_2$. As we said in the beginning of Section \ref{preliminaries} all the sets considered in this paper are supposed to be equipped with the Euclidean induced metric. Whenever we consider the inner distance, we emphasize it clearly.

\begin{definition}[See \cite{BirbrairM:2000}] A subset $X\subset\R^n$ is called {\bf Lipschitz normally embedded} if there exists $\lambda >0$ such that
$$d_X(x_1,x_2)\le \lambda \|x_1-x_2\|$$
for all $x_1,x_2\in X$.
\end{definition}

\begin{proposition}
If a closed unbounded subset $X\subset\R^n$ is Lipschitz regular at infinity, then there exists a compact $K\subset\R^n$ such that each connected component of $X\setminus K$ is Lipschitz normally embedded.
\end{proposition}

\begin{proof}Let $X\subset\R^n$ be  a closed and unbounded subset. Let us suppose that $X$ is Lipschitz regular
at infinity, that is, there exist compact subsets $K_1\subset\R^k$ and $K_2\subset\R^n$ and a bi-Lipschitz homeomorphism $\psi \colon \R^k\setminus K_1\rightarrow X\setminus K_2$. Without loss generality, one can suppose that $K_1$ is a  Euclidean closed ball. Let us denote $Y=\R^k\setminus K_1$, $Z=X\setminus K_2$.

First, let us  suppose that $k>1$. Since there are positive constant $\lambda_1 <\lambda_2$ such that:
$$ \lambda_1\|p-q\|\leq \|\psi(p)-\psi(q)\|\leq \lambda_2 \|p-q\|, \, \forall p,q\in Y,$$
it follows that
$$\lambda_1d_Y(p,q)\leq d_Z(\psi(p),\psi(q))\leq \lambda_2 d_Y(p,q), \, \forall p,q\in Y.$$
On the other hand,  $d_Y(p,q)\leq \pi|p-q|$, for all $p,q\in Y$ and, by those inequalities above, it follows that
$$d_Z(\psi(p),\psi(q))\leq\frac{\lambda_2\pi}{\lambda_1}\|\psi(p)-\psi(q)\|, \, \forall \ p,q\in Y.$$
Therefore,
$$d_Z(x,y)\leq\frac{\lambda_2\pi}{\lambda_1}\|x-y\|, \, \forall \ x,y\in Z.$$
In other words, $Z$ is Lipschitz normally embedded.

In the case where $k=1$, $\R\setminus K_1$ has two connected components which we denote  by $Y_1$ and $Y_2$. Actually, $Y_1$ and $Y_2$ are half-lines and, therefore, for each $i=1,2$, $d_{Y_i}(p,q)\leq |p-q|$, for all $p,q\in Y_i$. Likewise as it was done above, we have
$$d_{Z_i}(x,y)\leq\frac{\lambda_2}{\lambda_1}\|x-y\|, \, \forall \ x,y\in Z_i,$$
where $Z_1=\psi(Y_1)$ and $Z_2=\psi(Y_2)$ are the connected components of $Z$. Hence, the proposition is proved.
\end{proof}

Let us finish this section pointing out the following result which we are going to use in the proof of Theorem \ref{main theorem}

\begin{corollary}\label{neighborhood}
Let $X\subset\C^n$ be a complex algebraic subset. If $X$ is Lipschitz regular at infinity, then there exists a compact subset $K\subset\C^n$ such that $X\setminus K$ is Lipschitz normally embedded.
\end{corollary}

\section{Tangent cone at infinity}\label{cones}
Let us start this section recalling two well known results about semialgebraic sets, namely, the Monotonicity Theorem and Curve Selection Lemma.

\begin{lemma}[Theorem 1.8 in \cite{HaP:2017}]\label{monotonicity}
Let $f\colon (a,b)\to \R$ be a semialgebraic function. Then, there are $a=a_0<a_1<...<a_k=b$ such that, for each $i=0,...,k-1$, the restriction $f|_{(a_i,a_{i+1})}$ is analytic and either constant, strictly increasing or strictly decreasing.
\end{lemma}

\begin{lemma}[Theorem 2.5.5 in \cite{BochnakCR:1998}]\label{selection_curve}
Let $X$ be a semialgebraic subset of $\R^n$ and $x\in\R^n$ being a non-isolated point of $\overline{X}$. Then, there
exists a continuous semialgebraic mapping $\gamma\colon [0,1]\to \R^n$ such that $\gamma(0)=x$ and $\gamma((0,1])\subset X$.
\end{lemma}

\begin{definition}
Let $X\subset \R^m$ be an unbounded subset. We say that $E\subset \R^{m}$ is {\bf a tangent cone of $X$ at infinity} if there are a sequence $\{t_j\}_{j\in \N}$ of positive real numbers and a subset $D\subset \mathbb{S}^{m-1}$ such that $\lim\limits_{j\to \infty} t_j=+\infty $, $E=\{tv;\, v\in D\mbox{ and }t\geq 0\}$ and $D=\lim \frac{1}{t_j} X_{t_j}$, where $X_{t_j}=X\cap \{x\in \R^m; \|x\|=t_j\}$ for all $j\in \N$ and the limit is the Hausdorff's limit. When $X$ has a unique tangent cone at infinity, we denote it by $C_{\infty }(X)$ and we call $C_{\infty }(X)$ {\bf the tangent cone of $X$ at infinity}.
\end{definition}

\begin{proposition}\label{selection_lemma}
Let $Z\subset \R^n$ be an unbounded semialgebraic set and let $E$ be a tangent cone of $Z$ at infinity. A vector $v\in\R^n$ belongs to $E$ if and only if there exists a continuous semialgebraic  curve $\gamma\colon (\varepsilon ,+\infty )\to Z$  such that $\lim\limits _{t\to +\infty }|\gamma(t)|=+\infty $ and $\gamma(t)=tv+o_{\infty }(t),$ where $g(t)=o_{\infty }(t)$ means $\lim\limits _{t\to +\infty }\frac{g(t)}{t}=0$. 
\end{proposition}
\begin{proof}
Let us consider the semialgebraic mapping $\phi:\R^n\setminus\{0\}\to \R^n\setminus\{0\}$ given by $\phi(x)=\frac{x}{\|x\|^2}$ and denote $X=\phi(Z\setminus \{0\})$.  Since $Z$ is an unbounded set,  $0\in \overline{X}$. Let $\rho:\mathbb{S}^{n-1}\times [0,+\infty )\to \R^n\setminus \{0\}$ be the mapping given by $\rho(x,t)=tx$. We see that $\rho|_{\mathbb{S}^{n-1}\times (0,+\infty )}:\mathbb{S}^{n-1}\times (0,+\infty )\to \R^n\setminus \{0\}$ is a semialgebraic homeomorphhism with inverse mapping $\rho^{-1}:\R^n\setminus \{0\}\to \mathbb{S}^{n-1}\times (0,+\infty )$ given by $\rho^{-1}(x)=(\frac{x}{\|x\|},\|x\|)$. Therefore, the set $Y=\rho^{-1}(X)\subset \mathbb{S}^{n-1}\times [0,+\infty )$ and $\overline{Y}$ are semialgebraic sets.

Since $E$ is a tangent cone of $Z$ at infinity, there are a sequence $\{t_j\}_{j\in \N}$ of positive real numbers and a subset $D\subset \mathbb{S}^{m-1}$ such that $\lim\limits_{j\to \infty} t_j=+\infty $, $E=\{tv;\, v\in D\mbox{ and }t\geq 0\}$ and $D=\lim \frac{1}{t_j} X_{t_j}$, where $X_{t_j}=X\cap \{x\in \R^m; \|x\|=t_j\}$ for all $j\in \N$. Let $w\in E$ and we are going to consider two cases:

\bigskip

\noindent 1) Case $w\not=0$. By definition of $E$, there exists a sequence  $\{z_k\}_{k\in\N}\subset X\setminus\{0\}$ such that $\|z_k\|=t_k$ for all $k\in \N$ and $\lim \limits_{k\to \infty }\frac{z_k}{s_k}=w$, where $s_k=\frac{t_k}{\|w\|}$ for all $k\in \N$. Thus, for each $k\in \N$, let us define $x_k=\phi(z_k)$. In this case,  $v:=\lim\limits _{k\to \infty } s_k\cdot x_k= \frac{w}{\|w\|^2}$. In particular, $\lim\limits _{k\to \infty } \frac{x_k}{\|x_k\|}=\frac{w}{\|w\|}=\frac{v}{\|v\|}$ and $u=(\frac{v}{\|v\|},0)\in \overline{Y}$. Then by Curve Selection Lemma (Lemma \ref{selection_curve}), there exists a continuous semialgebraic curve $\beta:[0,\delta)\to \overline{Y}$ such that $\beta(0)=u$ and $\beta((0,\delta ))\subset Y$. By writing $\beta(t)=(x(t),s(t))$, we get $s:[0,\delta )\to \R$ is a semialgebraic and non-constant function such that $s(0)=0$ and $s(t)>0$ if $t\in (0,\delta )$. By Lemma \ref{monotonicity}, one can suppose that $s$ is analytic in the domain $(0,\delta)$ and  strictly increasing. Hence, $s:[0,\delta/2]\to [0,\delta' ]$ is a semialgebraic homeomorphism, where $\delta' =s(\frac{\delta}{2})$. Let us define $\alpha:[0,r)\to \overline{X}$ by
$$
\alpha(t)=\rho\circ \beta\circ s^{-1}(t\|v\|)=\rho(x(s^{-1}(t\|v\|)),s(s^{-1}(t\|v\|)))=t\|v\|x(s^{-1}(t\|v\|)),
$$
where $r=\min\{\frac{\delta'}{\|v\|},\delta'\}$. Therefore,
$$
\lim\limits _{t\to 0^+}\frac{\alpha(t)}{t}=\lim\limits _{t\to 0^+}\frac{t\|v\|x(s^{-1}(t))}{t}=\lim\limits _{t\to 0^+}\|v\|x(s^{-1}(t))=\|v\|x(0)=v,
$$
and $\alpha(t)=tv+o(t)$. Finally, by defining $\gamma:(\frac{1}{r},+\infty )\to Z$ in this way $\gamma(t)=\phi^{-1}(\alpha(\frac{1}{t}))$, we get
\begin{eqnarray*}
\gamma(t) &=& \frac{\frac{1}{t}v+o(\frac{1}{t})}{\|\frac{1}{t}v+o(\frac{1}{t})\|^2}=t\textstyle{\frac{v}{\|v\|^2}}+o_{\infty }(t)\\
		  &=& tw+o_{\infty }(t).
\end{eqnarray*}
Since $\gamma$ is a composition of semialgebraic mappings,  $\gamma$ is a semialgebraic mapping as well.

\bigskip

\noindent 2) Case $w=0$. In this case,  let $\{x_k\}_{k\in\N}\subset Z$ be a sequence satisfying $\|x_k\|=t_k$ for all $k\in \N$ (this sequence exists, because $Z$ is unbounded). Thus, $\{\frac{x_k}{\|x_k\|}\}_{k\in\N}$ is a convergent sequence. Let $v\in  E$ be the limit of this sequence, i.e., $\lim\limits _{k\to \infty }\frac{x_k}{\|x_k\|}=v$. Likewise as it was done in the Case 1, one can show that there exists a continuous semialgebraic curve $\gamma\colon (\varepsilon,+\infty )\to Z$ such that $\gamma(t)=tv+o_{\infty }(t)$. Let us define $\widetilde\gamma\colon (\varepsilon^{2}, +\infty )\to Z$ by $\widetilde{\gamma}(t)=\gamma(t^{\frac{1}{2}})$. Thus, we have $\widetilde{\gamma}(t)=o_{\infty }(t)=tw+o_{\infty }(t)$.

\bigskip

We have proved that  $E\subset\{v\in\R^n;\, \exists \gamma:(\varepsilon,+\infty )\to Z$ semialgebraic s.t.  $\lim\limits _{t\to +\infty }|\gamma(t)|=+\infty$ and $\gamma(t)=tv+o_{\infty }(t)\}$. Since the other inclusion is obvious, we have finished the proof.
\end{proof}
In particular, Proposition \ref{selection_lemma} tell us that if $Z\subset \R^n$ is an unbounded semialgebraic set, then $Z$ has a unique tangent cone at infinity.

The next result is a kind of version at infinity of  Theorem 3.2 in \cite{Sampaio:2016}, where the second named author of this paper proved that  bi-Lipschitz homeomorphic subanalytic subsets have  bi-Lipschitz homeomorphic tangent cones.

\begin{theorem}\label{tg_cones}
Let $A\subset \R^m$ and $B\subset\R^n$ be unbounded semialgebraic subsets. If $A$ and $B$ are bi-Lipschitz homeomorphic at infinity, then their tangent cones at infinity $C_{\infty }(A)$ and $C_{\infty }(B)$ are bi-Lipschitz homeomorphic.
\end{theorem}

\begin{proof}
Let $\widetilde K_1\subset\R^m$ and $\widetilde K_2\subset \R^n$ be compact subsets such that there exists a bi-Lipschitz homeomorphism $\phi\colon A\setminus \widetilde K_1\to B\setminus \widetilde K_2$. Let us denote $X=A\setminus \widetilde K_1$ and $Y=B\setminus  \widetilde K_2$. By taking $\R^N=\R^m\times\R^n$ and doing the following identifications:
$$X \leftrightarrow X\times\{ 0\} \ \mbox{ and } \ Y \leftrightarrow \{ 0\}\times Y $$ one can suppose that $X,Y\subset\R^N$ and there exists a bi-Lipschitz map $\varphi\colon \R^N\to\R^N$ such that $\varphi(X)=Y$ (see Lemma 3.1 in \cite{Sampaio:2016}). Let $K>0$ be a constant such that
\begin{equation}
\frac{1}{K}\|x-y\|\leq \|\varphi(x)-\varphi(y)\|\leq K\|x-y\|, \quad \forall x,y\in \R^N.
\end{equation}

For each $k\in\N$, let us  define the mappings $\varphi_k,\psi_k:\R^N\to \R^N$ given by $\varphi_k(v)=\frac{1}{k}\varphi(kv)$ and $\psi_k(v)=\frac{1}{k}\varphi^{-1}(kv)$. For each integer $m\geq 1$, let us define $\varphi_{k,m}:=\varphi_k|_{\overline{B}_m}:\overline{B}_m\to \R^N$ and $\psi_{k,m}:=\psi_k|_{\overline{B}_{mK}} :\overline{B}_{mK}\to \R^N$, where $\overline{B}_r$ denotes the Euclidean closed ball of radius $r$ and with center at the origin in $\R^N$.
Since
$$
\frac{1}{K}\|x-y\|\leq \|\varphi_{k,1}(x)-\varphi_{k,1}(y)\|\leq K\|x-y\|, \quad \forall x,y\in \overline{B}_1,\,\, \forall k\in \N
$$
and
$$
\frac{1}{K}\|u-v\|\leq \|\psi_{k,1}(u)-\psi_{k,1}(v)\|\leq K\|u-v\|, \quad u,v\in \overline{B}_K,\,\, \forall k\in \N,
$$
there exist a subsequence $\{k_{j,1}\}_{j\in \N}\subset \N$ and Lipschitz mappings $d\varphi_1:\overline{B}_1\to \R^N$ and $d\psi_1: \overline{B}_K\to \R^N$ such that $\varphi_{k_{j,1},1}\rightrightarrows d\varphi_1$ uniformly on $\overline{B}_1$ and $ \psi_{k_{j,1},1}\rightrightarrows d\psi_1$ uniformly on $\overline{B}_K$ (notice that $\{\varphi_{k,1}\}_{k\in \N}$ and $\{\psi_{k,1}\}_{k\in \N}$ have uniform Lipschitz constants). Furthermore, it is clear that
$$
\frac{1}{K}\|u-v\|\leq \|d\varphi_1(u)-d\varphi_1(v)\|\leq K\|u-v\|, \quad \forall u,v\in \overline{B}_1
$$
and
$$
\frac{1}{K}\|z-w\|\leq \|d\psi_1(z)-d\psi_1(w)\|\leq K\|z-w\|, \quad \forall z,w\in \overline{B}_K.
$$
Likewise as above, for each $m>1$, we have
$$
\frac{1}{K}\|x-y\|\leq \|\varphi_{k,m}(x)-\varphi_{k,m}(y)\|\leq K\|x-y\|, \quad x,y\in \overline{B}_m,\,\, \forall k\in \N
$$
and
$$
\frac{1}{K}\|u-v\|\leq \|\psi_{k,m}(u)-\psi_{k,m}(v)\|\leq K\|u-v\|, \quad u,v\in \overline{B}_{mK},\,\, \forall k\in \N.
$$
Therefore, for each $m>1$, there exist a subsequence $\{k_{j,m}\}_{j\in \N}\subset\{k_{j,m-1}\}_{j\in \N}$ and Lipschitz mappings $d\varphi_m\colon \overline{B}_m\to \R^N$ and $d\psi_m\colon \overline{B}_{mK}\to \R^N$ such that $\varphi_{k_{j,m},m}\rightrightarrows d\varphi_m$ uniformly on $\overline{B}_m$ and $ \psi_{k_{j,m},m}\rightrightarrows d\psi_m$ uniformly on $\overline{B}_{mK}$ with $d\varphi_m|_{\overline{B}_{m-1}}=d\varphi_{m-1}$ and $d\psi_m|_{\overline{B}_{(m-1)K}}=d\psi_{m-1}$. Furthermore,
\begin{equation}\label{diflipum}
\frac{1}{K}\|u-v\|\leq \|d\varphi_m(u)-d\varphi_m(v)\|\leq K\|u-v\|, \quad \forall u,v\in \overline{B}_m
\end{equation}
and
\begin{equation}\label{diflipdois}
\frac{1}{K}\|z-w\|\leq \|d\psi_m(z)-d\psi_m(w)\|\leq K\|z-w\|, \quad \forall z,w\in \overline{B}_{mK}.
\end{equation}

Let us define $d\varphi,d\psi:\R^N\to \R^N$ by $d\varphi(x)=d\varphi_m(x)$, if $x\in \overline{B}_m$ and $d\psi(x)=d\psi_m(x)$, if $x\in \overline{B}_{mK}$ and, for each  $j\in \N$, let $t_j=n_j=k_{j,j}$.

\bigskip

\noindent {\it Claim 1.} $\varphi_{n_j}\rightarrow d\varphi$ and $\psi_{n_j}\rightarrow d\psi$ uniformly on compact subsets of $\R^N$.

\bigskip

Let $F\subset \R^N$ be a compact subset. Let us take  $m\in \N$ such that $F\subset \overline{B}_m\subset \overline{B}_{mK}$. Thus, $\{n_j\}_{j>m}$ is a subsequence of $\{k_{j,m}\}_{j\in \N}$ and, since  $\varphi_{k_{j,m},m}\rightrightarrows d\varphi_m$ uniformly on $\overline{B}_m$ and $ \psi_{k_{j,m},m}\rightrightarrows d\psi_m$ uniformly on $\overline{B}_{mK}$, it follows that $\varphi_{n_j}\rightrightarrows d\varphi$ and $\psi_{n_j}\rightrightarrows d\psi$ uniformly on $F$.

\bigskip

\noindent {\it Claim 2.} $d\varphi:\R^N\to \R^N$ bi-Lipschitz homeomorphism and $d\psi= (d\varphi)^{-1}$.

\bigskip

It follows from inequalities (\ref{diflipum}) and (\ref{diflipdois}) that $d\varphi,d\psi:\R^N\to \R^N$ are Lipschitz mappings. Therefore, it is enough to show that $d\psi= (d\varphi)^{-1}$. In order to do that, let $v\in \R^N$ and $w=d\varphi(v)=\lim \limits_{j\to \infty }\frac{\varphi(t_jv)}{t_j}$. Thus,
$$
\begin{array}{lllll}
\|d\psi(w)-v\|&=&\bigg\|\lim \limits_{j\to \infty }\frac{\psi(t_jw)}{t_j} - v\bigg\| &=&\lim \limits_{j\to \infty }\bigg\|\frac{\psi(t_jw)}{t_j} - \frac{t_jv}{t_j}\bigg\|\\
				    &=&\lim \limits_{j\to \infty }\frac{1}{t_j}\bigg\|\psi(t_jw) - t_jv\bigg\| &=&\lim \limits_{j\to \infty }\frac{1}{t_j}\bigg\|\psi(t_jw) - \psi(\varphi(t_jv))\bigg\|\\
				    &\leq &\lim \limits_{j\to \infty }\frac{K}{t_j}\bigg\|t_jw - \varphi(t_jv)\bigg\| &\leq &\lim \limits_{j\to \infty }K\bigg\|w - \frac{\varphi(t_jv)}{t_j}\bigg\|\\
					&=&0.& &
\end{array}
$$

Then, $d\psi(w)=d\psi(d\varphi(v))=v$, for all $v\in \R^N$, i.e., $d\psi\circ d\varphi={\rm id}_{\R^N}$. Analogously, one can show that $d\varphi\circ d\psi={\rm id}_{\R^N}$.

\bigskip

\noindent {\it Claim 3.} $d\varphi(C_{\infty }(X))= C_{\infty }(Y)$.

\bigskip

By  Claim 2, it is enough to verify that $d\varphi(C_{\infty }(X))\subset C_{\infty }(Y)$. In order to do that, let  $v\in C_{\infty }(X)$. Then, there is $\alpha\colon (\varepsilon, \infty)\to X$ such that $\alpha(t)=tv+o_{\infty }(t)$. Thus, $\varphi(\alpha(t))=\varphi(tv)+o_{\infty }(t)$, since $\varphi$ is a Lipschitz mapping. However, $\varphi(t_jv)=t_jd\varphi(v)+o_{\infty }(t_j)$ and then
$$
d\varphi (v)=\lim \limits_{j\to \infty }\varphi_{n_j}(v)=\lim \limits_{j\to \infty }\frac{\varphi(t_jv)}{t_j}=\lim \limits_{j\to \infty }\frac{\varphi(\alpha(t_j))}{t_j}\in C_{\infty }(Y).
$$

Therefore, $d\varphi:C_{\infty }(X)\to C_{\infty }(Y)$ is a bi-Lipschitz homeomorphism. We finish the proof by remarking that $C_{\infty }(A)=C_{\infty }(X)$ and $C_{\infty }(B)=C_{\infty }(Y)$.

\end{proof}

Let us finish this section pointing out the following result which we are going to use in the proof of Theorem \ref{main theorem}

\begin{corollary}
Let $X\subset\C^n$ be a complex algebraic subset. If $X$ is Lipschitz regular
at infinity, then the tangent cone of $X$ at infinity is Lipschitz regular at infinity.
\end{corollary}

\section{Proof of Theorem \ref{main theorem}}\label{main}

Let us begin this section by recalling some basic facts about degree of complex algebraic sets. More precisely, we are going to see that affine linear subspaces of $\C^n$ are characterized among algebraic subsets of $\C^n$ by being those  of degree 1. In such a direction, let $\iota\colon\C^n\hookrightarrow \mathbb{P}^n$ be the  embedding given by $\iota(x_1,...,x_n)=[1:x_1:...:x_n]$ and let $p\colon \C^{n+1}\setminus\{0\}\to \mathbb{P}^n$ be the projection mapping given by $p(x_0,x_1,\dots,x_n)=[x_0:x_1:...:x_n]$.

\begin{remark}
 {\rm Let $A$ be an algebraic set in $\mathbb{P}^n$ and $X$ be an algebraic set in $\C^n$. Then $\widetilde{A}=p^{-1}(A)\cup \{0\}$ is a homogeneous complex algebraic set in $\C^{n+1}$ and the closure $\overline{\iota(X)}$ of $\iota(X)$ in $\mathbb{P}^n$ is an algebraic set in $\mathbb{P}^n$.}
\end{remark}
\begin{definition}
 Let $A$ be an algebraic set in $\mathbb{P}^n$. We define {\bf the degree of} $A$ by ${\rm deg}(A)=m(\widetilde{A},0)$, where $m(\widetilde{A},0)$ is the multiplicity of $\widetilde{A}$ at $0\in\C^{n+1}$.
\end{definition}
\begin{definition}
 Let $X$ be a complex algebraic set in $\C^n$. We define {\bf the degree of} $X$ by $ {\rm deg}(X) ={\rm deg}(\overline{\iota(X)})$.
\end{definition}

\begin{proposition}\label{linear_prop}
Let $X\subset \mathbb{\C}^n$ be an algebraic subset. Then,  ${\rm deg}(X)=1$ if and only if $X$ is an affine linear subspace of $\mathbb{C}^n$.
\end{proposition}

\begin{proof}
 Let $A=\overline{\iota(X)}$ be the closure of $\iota(X)$ in $\mathbb{P}^n$. By definition, ${\rm deg}(X)={\rm deg}(A)$ and ${\rm deg}(A)=m(\widetilde{A},0)$, where $\widetilde{A}=p^{-1}(A)\cup \{0\}$. Thus, ${\rm deg}(X)=1$ if and only if $m(\widetilde{A},0)=1$. However, as $\widetilde{A}$ is a homogeneous complex algebraic set in $\C^{n+1}$, $m(\widetilde{A},0)=1$ if and only if $\widetilde{A}$ is a complex linear subspace. In order to finish the proof, we remark that $\widetilde{A}$ is a complex linear subspace in $\C^{n+1}$ if and only if $A$ is a complex projective plane in $\mathbb{P}^n$.
\end{proof}

From now on, we are going to prove of the main results of the paper.

\begin{theorem}\label{proposition A}
Let $X\subset\C^n$ be a pure $d$-dimensional algebraic subset such that $C_{\infty}(X)$ is a complex linear subspace of $\C^n$. If there exists a compact subset $K\subset\C^n$ such  that $X\setminus K$ is Lipschitz normally embedded, then $X$ is an affine linear subspace of $\C^n$.
\end{theorem}
\begin{proof} Since $C_{\infty }(X)$ is a linear subspace of $\C^n$, one can consider the orthogonal projection $\pi\colon \C^n\rightarrow C_{\infty }(X).$ Let us choose linear coordinates $(x,y)$ in $\C^n$ such that $$C_{\infty }(X)=\{(x,y)\in\C^n;\,y=0\}.$$ Notice that the restriction of the orthogonal projection $\pi$ to $X$ has the following properties:
\begin{enumerate}
\item [1)] There exist a compact subset $K\subset \C^n$ and a constant $C>0$ such that $\|y\|\leq C\|x\|$ for all $(x,y)\in X\setminus K$.
\item [2)] If $\gamma :(\varepsilon,\infty )\rightarrow X$ is an arc such that $\lim\limits_{t\to +\infty}\|\gamma(t)\|=+\infty $ and $\pi\circ\gamma(t)=tv+o_{\infty}(t)$, then $\gamma(t)=tv+o_{\infty}(t)$.
\end{enumerate}

Indeed, if 1) is not true, there exists a sequence $\{(x_k,y_k)\}\subset X$ such that $\lim\limits_{k\to+\infty}\|(x_k,y_k)\|=+\infty$ and $\|y_k\|> k\|x_k\|$. Thus, up to a subsequence,  one can suppose that $\lim\limits_{k\to+\infty}\frac{y_k}{\|y_k\|}=y_0$. Since $\frac{\|x_k\|}{\|y_k\|}< \frac{1}{k}$,  $(0,y_0)\in C_{\infty }(X)$, which is a contradiction, because $y_0\not=0$. Therefore, 1) is true.

Now, in order to prove 2), let us write $\gamma(t)=(x(t),y(t))$. By 1), there exists $t_0>0$ such that  $\|y(t)\|\leq C\|x(t)\|$ for all $t\geq t_0$, since $\lim\limits_{t\to +\infty}\|\gamma(t)\|=+\infty$.  Thus, since $\frac{x(t)}{t}$ is bounded, $\frac{y(t)}{t}$ is bounded. Let us suppose that $y(t)\not=o_{\infty}(t)$. Then, there exist a sequence $\{t_k\}_{k\in \N}$ and $r>0$ such that $t_k\to +\infty$ and $\frac{\|y(t_k)\|}{t_k}\geq r$ for all $k$. Since $\left\{\frac{y(t_k)}{t_k}\right\}_{k\in \N}$ is bounded, up to a subsequence, one can suppose  that $\lim\limits_{k\to +\infty}\frac{y(t_k)}{t_k}=y_0$. Therefore, $\lim\limits_{k\to +\infty}\frac{\gamma(t_k)}{t_k}=(v',y_0)\in C_{\infty }(X)$, where $v=(v',0)$. However, this is a contradiction, since $\|y_0\|\geq r>0$ and this implies that $y_0\not=0$. Then, $y(t)=o_{\infty}(t)$ and, therefore, $\gamma(t)=tv+o_{\infty}(t)$.

Let $L=\pi^{-1}(0)$. One can see that $\overline{\iota(X)}\cap\overline{\iota(L)}\cap (\mathbb{P}^n\setminus \iota( \C^n))=\emptyset $. Therefore, $\pi|_{ X}\colon X\rightarrow C_{\infty }(X)$ is a ramified cover with degree equal to ${\rm deg} (X)$ (see \cite{Chirka:1989}, Corollary 1 in the page 126). Moreover, the ramification locus of $\pi|_X$ is a codimension $\geq 1$ complex algebraic subset $\Sigma$ of the linear space $ C_{\infty }(X)$.

Let us suppose that the degree ${\rm deg} (X)$ is greater than 1. Since $\Sigma$ is a codimension $\geq 1$ complex algebraic subset of the space $C_{\infty }(X)$, there exists a unit tangent vector $v_0\in C_{\infty }(X)\setminus C_{\infty }(\Sigma)$,
where $C_{\infty }(\Sigma) $ is the tangent cone at infinity of $\Sigma$.

Since $v_0$ is not tangent to $\Sigma$ at infinity, there exist positive real numbers $\lambda$ and $R$ such that
$$C_{\lambda,R}=\{v\in C_{\infty }(X); \ \|v-tv_0\|< \lambda t, \ \forall \ t>R \}$$
does not intersect the set $\Sigma$. Since we have assumed that the degree ${\rm deg} (X)\geq 2$, we have at least two different liftings $\gamma_1(t)$ and $\gamma_2(t)$ of the half-line $r(t)=tv_0$,  i.e. $\pi(\gamma_1(t))=\pi(\gamma_2(t))=tv_0$.  Since $\pi$ is the orthogonal projection on $C_{\infty }(X)$ and the vector $v_0$ is the unit tangent vector at infinity to the images
$\pi\circ \gamma_1$ and $\pi\circ \gamma_2$, then $v_0$ is the tangent vector at infinity to the arcs $\gamma_1$ and $\gamma_2$. By construction, we have ${\rm dist}(\gamma_i(t),{\pi|_ X}^{-1}(\Sigma))\geq \lambda t$ for $i=1,2$, where by dist we mean the Euclidean distance.

On the other hand, any path in $X$ connecting $\gamma_1(t)$ to $\gamma_2(t)$ is the lifting of a loop, based at the point $tv_0$, which is not contractible in $C_{\infty }(X)\setminus {\Sigma}$. Thus, the length of such a path must be at least  $2\lambda t$. It implies that the inner distance, $\rm{d_{X}}(\gamma_1(t),\gamma_2(t))$, in $X$, between $\gamma_1(t)$ and $\gamma_2(t)$, is at least $2\lambda t$. But, since $\gamma_1(t)$ and $\gamma_2(t)$ are tangent at infinity, that is,
$$\frac{\|\gamma_1(t)-\gamma_2(t)\|}{t}\to 0 \ \mbox{ as } \ t\to +\infty ,$$
and $\lambda>0$, we obtain that $X$ is not Lipschitz normally embedded at infinity. Otherwise there will be $\widetilde C >0$ and a compact subset $K\subset\C^n$ such that:
$$d_X(x_1,x_2)\leq \widetilde  C \|x_1-x_2\| \quad \mbox{for all} \quad  x_1,x_2\in X\setminus K,$$
hence:
\begin{eqnarray*}
2\lambda &\leq & \frac{\rm{d_{X}}(\gamma_1(t),\gamma_2(t))}{t} \\
&\leq &\widetilde C \frac{\|\gamma_1(t)-\gamma_2(t)\|}{t}\to 0 \ \ \mbox{as} \ t\to +\infty,
\end{eqnarray*}
which is a contradiction. We have concluded that ${\rm deg}(X)=1$ and, by Proposition \ref{linear_prop}, it follows that $X$ is an affine linear subspace.

\end{proof}

Notice that the assumptions in  Theorem \ref{proposition A} are sharp in the sense that, in order to get the same conclusion, none of those assumptions can be removed, as we can see bellow.
\begin{example}
The plane complex curve $X=\{(x,y)\in \C^2;\, y=x^2\}$ has linear tangent cone at infinity; $X$ is not an affine linear subset of $\mathbb{C}^2$. As another example, $Z=\{(x,y,z)\in \C^3;\, x^2+y^2+z^2=0\}$. We see that $Z$ is Lipschitz normally embedded; $Z$ is not an affine linear subset of $\C^3$.
\end{example}

As an application of Theorem \ref{proposition A}, we obtain a result like Theorem J in (\cite{GreeneW:1979}, p. 180) (see also \cite{SiuY:1977} and \cite{Itoh:1979}).

\begin{corollary}\label{analytic_version}
Let $X\subset\C^n$ be a pure $d$-dimensional analytic subset. Suppose $X$ has a unique tangent cone at infinity and it is a $d$-dimensional complex linear subspace of $\C^n$. If $X$ is closed and Lipschitz Normally Embedded at infinity, then $X$ is an affine linear subspace of $\C^n$.
\end{corollary}
\begin{proof} Let us choose linear coordinates $(x,y)$ in $\C^n$ such that $C_{\infty }(X)=\{(x,y)\in\C^n;\,y=0\}$ and consider the orthogonal projection $\pi\colon \C^n\rightarrow C_{\infty }(X).$ 

\noindent {\bf Claim.} There exist positive constants $C$ and $\rho$ such that $X\subset \{(x,y);\|y\|< C\|x\|\}\cup B_{\rho}$.

Indeed, since $C_{\infty }(X)=\{(x,y)\in\C^n;\,y=0\}$, by the proof of Theorem \ref{proposition A}, there exist a positive constant $C$ and a compact $K\subset \C^n$ such that $\|y\|< C\|x\|$ for all $(x,y)\in X\setminus K$. Let $\rho$ be a positive number such that $K\subset B_{\rho}$. Thus, we have that $X\subset \{(x,y);\|y\|< C\|x\|\}\cup B_{\rho}$.

Now, by Theorem 2 in (\cite{Chirka:1989}, page 77), $X$ is an algebraic set and the proof follows from Theorem \ref{proposition A}.
\end{proof}

Let us remark that the real version of Corollary \ref{analytic_version} does not hold true in general, as it is shown bellow.

\begin{example}
The set $X=\{(x,y)\in\R^2;\,y=\sin{x}\}$ is Lipschitz Normally embedded, since it is bi-Lipschitz homeomorphic to $\R$ and, moreover, $X$ has a unique tangent cone at infinity with $C_{\infty }X=\{(x,y)\in\R^2;\,y=0\}$. However, $X$ is not an algebraic set and, in particular, it is not a linear subspace of $\R^2$.
\end{example}

Let $X\subset\C^n$ be a complex algebraic subset. Let $\mathcal{I}(X)$ be the ideal of $\C[x_1,...,x_n]$ given by the polynomials which vanishes on $X$. For each $f\in\C[x_1,...,x_n]$, let us denote by $f^*$ the homogeneous polynomial composed of the monomials in $f$ of maximum degree.

\begin{proposition}[Theorem 1.1 in \cite{LeP:2016}]
Let $X\subset\C^n$ be a complex algebraic subset. Then, $C_{\infty }(X)$ is the affine variety $V(\langle f^*;\, f\in \mathcal{I}(X)\rangle) $.
\end{proposition}

It follows from above proposition that the tangent cone at infinity of a complex algebraic subset of $\C^n$ is a homogeneous complex algebraic subset and, therefore, it is a complex cone in the following sense: an algebraic subset of $\C^n$ is called a \emph{complex cone} if it is a union of one-dimensional complex linear subspaces of $\C^n$. The next result was proved by David Prill in \cite{Prill:1967}:

\begin{lemma}[Theorem in \cite{Prill:1967}] \label{prill} Let $V\subset \C^n$ be a complex cone. If $0\in V$ has a neighborhood homeomorphic to a Euclidean ball, then $V$ is a linear subspace of $\C^n$
\end{lemma}

At this moment, we are ready to give a proof of Theorem \ref{main theorem}
\begin{proof} [Proof of Theorem \ref{main theorem}]

Let us suppose that the complex algebraic subset $X\subset\C^n$ is Lipschitz regular at infinity. Thus, let $h\colon U\rightarrow \R^N\setminus B$ be a bi-Lipschitz homeomorphism, where is an open neighborhood $U$ of the infinity in $X$ (i.e. a complement of a compact subset in X) and $B\subset\R^N$ is a closed Euclidean ball centered at the origin $0\in\R^N$. Let $C_{\infty }(X)$ be the tangent cone at infinity of $X$.
It comes from Theorem \ref{tg_cones} that there exists a bi-Lipschitz homeomorphism $dh\colon C_{\infty }(X) \rightarrow C_{\infty }(\R^N\setminus B)=\R^N$. In particular, $C_{\infty }(X)$ is a topological manifold. By Prill's Theorem (Lemma \ref{prill}), it follows that $C_{\infty }(X)$ is a complex linear subspace of $\C^n$. By Corollary \ref{neighborhood}, there exists a compact subset $K$ such that $X\setminus K$ is Lipschitz normally embedded and, by Theorem \ref{proposition A}, it follows that $X$ is an affine linear subspace of $\C^n$.
\end{proof}




\end{document}